\documentclass[12pt]{amsart}
\def\cal#1{\mathcal{#1}}

          \usepackage{amssymb}
          \usepackage{amsmath}
          \usepackage{amsthm}
          \usepackage{amsfonts}
          \usepackage[english]{babel}
          \usepackage[utf8]{inputenc}
          \usepackage{enumerate}
          \usepackage{graphicx}
          \usepackage{url}
          \usepackage{color}

\usepackage[numbers]{natbib}
\newcommand{\comment}[1]{}

\def\indd#1{{\bf 1}_{\{#1\}}}
\def\indzd#1#2{\{#1_#2\}_{#2\in \Zd}}
\def\indn#1{\{#1_n\}_{n\in\N}}
\newcommand{\proba}{\mathbb P}
\newcommand{\esp}{{\mathbb E}}

\newcommand{\defe}{\mathrel{\mathop:}=}
\newcommand{\inv}{^{-1}}

\newcommand{\calB}{{\cal B}}

\newcommand{\filF}{{\cal F}}
\newcommand{\calG}{{\cal G}}

\newcommand{\calM}{{\cal M}}
\newcommand{\calN}{{\cal N}}

\def\B{{\mathbb B}}

\newcommand{\eqnh}{\begin{eqnarray*}}
\newcommand{\eqne}{\end{eqnarray*}}
\newcommand{\eqnhn}{\begin{eqnarray}}
\newcommand{\eqnen}{\end{eqnarray}}
\newcommand{\equh}{\begin{equation}}
\newcommand{\eque}{\end{equation}}

\def\summ#1#2#3{\sum_{#1 = #2}^{#3}}
\def\prodd#1#2#3{\prod_{#1 = #2}^{#3}}

\def\sif#1#2{\sum_{#1=#2}^\infty}
\def\sumZ#1{\sum_{#1\in\mathbb Z}}
\def\sumzd#1{\sum_{#1\in\mathbb Z^d}}

\def\bcapp#1#2#3{\bigcap_{#1=#2}^{#3}}

\def\topp#1{^{(#1)}}

\def\nn#1{{\left\|#1\right\|}}

\def\snn#1{\|#1\|}

\def\abs#1{\left|#1\right|}
\def\sabs#1{|#1|}

\def\ccbb#1{\left\{#1\right\}}

\def\spp#1{(#1)}
\def\pp#1{\left(#1\right)} 
 
\def\bb#1{\left[#1\right]}

\def\mmid{\;\middle\vert\;}

\def\floor#1{\left\lfloor #1 \right\rfloor}

\def\vv#1{{\bf #1}}

\def\B{{\mathbb B}}

\def\mand{\mbox{ and }}
\def\qmand{\quad\mbox{ and }\quad}

\def\mwith{\mbox{ with }}
\def\qmwith{\quad\mbox{ with }\quad}
\def\mfa{\mbox{ for all }}
\def\qmfa{\quad\mbox{ for all }\quad}

\def\wt#1{\widetilde{#1}}

\def\weakto{\Rightarrow}

\def\limn{\lim_{n\to\infty}}

\def\limsupn{\limsup_{n\to\infty}}

\def\Z{{\mathbb Z}}
\def\Zd{{\mathbb Z^d}}
\def\R{{\mathbb R}}
\def\Rd{{\mathbb R^d}}

\def\N{{\mathbb N}}
\def\Nd{{\mathbb N^d}}

\usepackage{epsfig}
\usepackage{ifthen}

\newtheorem{Thm}{Theorem}[section]
\newtheorem{Lem}[Thm]{Lemma}
\newtheorem{Prop}[Thm]{Proposition}

\theoremstyle{definition}
\newtheorem{Rem}[Thm]{Remark}
\newtheorem{Def}[Thm]{Definition}
\newtheorem{Example}[Thm]{Example}

\numberwithin{equation}{section}

\title{}
\author{}

\begin{document}\sloppy
\title[Invariance principle for random fields]{An invariance principle for stationary random fields under Hannan's condition}

\author{Dalibor Voln{\'y}}
\address
{
Dalibor Voln\'y,
Laboratoire de Math\'ematiques Rapha\"el Salem,
Universit\'e de Rouen, 
76801, Saint Etienne du Rouvray, France.
}
\email{dalibor.volny@univ-rouen.fr}

\author{Yizao Wang}
\address
{
Yizao Wang,
Department of Mathematical Sciences,
University of Cincinnati,
2815 Commons Way,
Cincinnati, OH, 45221-0025.
}
\email{yizao.wang@uc.edu}

\begin{abstract}
We establish an invariance principle for a general class of stationary random fields indexed by $\Zd$, under  Hannan's condition generalized to $\Zd$. To do so we first establish a uniform integrability result for stationary orthomartingales, and second  we establish a coboundary decomposition for certain stationary random fields. At last, we obtain an invariance principle by developing an orthomartingale approximation. Our invariance principle improves known results in the literature, and particularly we require only finite second moment.
\end{abstract}


\maketitle
\section{Introduction}
Let $\{X_i\}_{i\in\Zd}$ be a stationary random field with zero mean and finite variance, and let $S_n$ be the partial sum with $n=(n_1,\dots,n_d)\in\Nd$
\[
S_n = \sum_{\vv 1\leq i \leq n}X_i,
\] and we are interested in the invariance principle of normalized partial sums in form of
\equh\label{eq:wip}
\ccbb{\frac{S_{\floor{n\cdot t}}}{|n|^{1/2}}}_{t\in[0,1]^d}\weakto \{\sigma\B_t\}_{t\in[0,1]^d},
\eque
where $n\cdot t = (n_1t_1,\dots,n_dt_d)$ and $|n| = \prodd q1dn_q$. We provide a sufficient condition for the above weak convergence to hold in $D[0,1]^d$, with the limiting random field being a Brownian sheet. 

The invariance principle for Brownian sheet has a long history, and people have investigated this problem from different aspects.
See for example \citet{berkes81strong,bolthausen82central,goldie86central,bradley89caution} for results under mixing conditions, \citet{basu79functional,morkvenas84invariance,nahapetian95billingsley,poghosyan98invariance} for results on multiparameter martingales, and \citet{dedecker98central,dedecker01exponential,elmachkouri13central,wang13new} for results on random fields satisfying projective-type assumptions. In particular, projective-type assumptions have been significantly developed  for invariance principles for stationary sequences ($d=1$). See for example \citet{wu05nonlinear,dedecker07weak}, among others, for some recent developments.
However, extending these criteria in one dimension to high dimensions is not a trivial problem. 

Our goal is to establish a random-field counterpart of the invariance principle for regular stationary sequences satisfying Hannan's condition \citep{hannan73central}. Hannan's condition consists of assuming, in dimension one, 
\equh\label{eq:Hannand=1}
\sum_{i\in\Z}\nn{P_0(X_i)}_2<\infty,
\eque
where $P_0(X_i) = \esp (X_i\mid\filF_0) - \esp(X_i\mid\filF_{-1})$ is the projection operator, with respect to 
 certain filtration $\{\filF_k\}_{k\in\Z}$ associated to the stationary sequence $\{X_k\}_{k\in\Z}$. Under Hannan's condition, if in addition the stationary sequence $\indn X$ is regular (i.e.~$\esp(X_0\mid\filF_{-\infty})= 0$ and $X_0$ is $\filF_\infty$-measurable), 
 then the invariance principle follows. \citet{hannan73central} first considered the invariance principle, under the assumption that $\{X_k\}_{k\in\Z}$ is adapted and weakly mixing. The general case for regular sequences was established by \citet[Corollary 2]{dedecker07weak}. The quenched invariance principle for adapted case has been established by \citet{cuny13quenched}. 
 
We first generalize the Hannan's condition~\eqref{eq:Hannand=1} to high dimension. For this purpose we need to extend the notion of the projection operator (Section~\ref{sec:prelim}). In particular, we focus on stationary random fields in form of
\equh\label{eq:Xi}
X_i = f\circ T_i(\{\epsilon_k\}_{k\in\Zd}), i\in\Zd,
\eque
where $f:\R^\Zd\to\R$ is a measurable function, $T_i$ is the shift operator on $\R^{\Zd}$ and $\{\epsilon_k\}_{k\in\Zd}$ is a collection of independent and identically distributed (i.i.d.) random variables. 

Our main result (Theorem~\ref{thm:wip}) states if $\esp X_{\vv0} = 0, \esp(X_{\vv0}^2)<\infty$, and the Hannan's condition holds
\[
\sum_{i\in\Zd}\nn{P_{\vv 0}X_i}_2<\infty,
\]
for some projection operator $P_{\vv 0}$ to be defined (see~\eqref{eq:P0} below),
  then the invariance principle~\eqref{eq:wip} holds.
  
We establish the invariance principle through an approximation by {\it orthomartingales}. As a consequence, this entails the central limit theorem in form of
\[
\frac{S_n}{|n|^{1/2}}\weakto\calN(0,\sigma^2).
\]
Our central limit theorem and invariance principle both improve results established in \citet{elmachkouri13central} and \citet{wang13new}. For the central limit theorem, our assumption on the weak dependence, the Hannan's condition, is weaker than theirs. Furthermore, to establish invariance principle we require only finite second moment instead of $2+\delta$ moment. However, we consider only rectanglular index sets as in \citet{wang13new}, while \citet{dedecker01exponential} and \citet{elmachkouri13central} consider more general index sets. 

The paper is organized as follows. The basic of orthomartingales is reviewed in Section~\ref{sec:prelim}. A uniform integrability result on orthomartingales is established in Section~\ref{sec:UI}, which immediately entails the tightness of stationary orthomartingales under finite second moment. Next, an orthomartingale coboundary decomposition is developed in Setion~\ref{sec:coboundary}. In this part our development is similar to the recent result by \citet{gordin09martingale}, who treated multiparameter reversed martingales and the corresponding coboundary decompositions. At last, the invariance principle under Hannan's condition is established in Section~\ref{sec:wip}. Comparison to related works are provided in Section~\ref{sec:discussions}.

\section{Notations and preliminaries}\label{sec:prelim}
We consider partial sums over rectangular sets. For this purpose we write $n = (n_1,\dots,n_d)\in\Nd$, $\N = \{1,2,\dots\}$, and by $n\to\infty$ we mean $n_q\to\infty$ for all $q=1,\dots,d$.  
Throughout, for elements in $\Rd$, operations (including $<$, $\leq$, $>$, $\geq$, $\pm$, $\wedge$, $\vee$) are defined in the coordinate-wise sense. We write $[m,n] = \{i\in\Zd: m\leq i\leq n\}$ for $m,n\in\Zd$ and  $[n] = [\vv 1,n]$. 
At last, we let $e_q = (0,\dots,0,1,0,\dots,0), q=1,\dots,d$ denote canonical unit vectors in $\Rd$. Throughout, let $(\Omega,\filF,\proba)$ be the underlying probability space.

We first review orthomartingales, essentially following \citet[Chapter 1.3]{khoshnevisan02multiparameter}, and introduce the projection operators. These two concepts are based on the notion of commuting filtrations. Specific examples via completely commuting transformations are given at the end.

\begin{Def}A collection of $\sigma$-algebras $\indzd\filF i$ is a filtration if $\filF_i\subset\filF_j$ for all $i,j\in\Zd, i\leq j$. It is commuting if in addition for all $k,\ell\in\Zd$ and  for all bounded $\filF_\ell$-measurable random variable $Y$, 
\equh\label{eq:commuting1}
\esp(Y\mid\filF_k) = \esp(Y\mid\filF_{k\wedge\ell}), \mbox{ almost surely.}
\eque
\end{Def}
For the sake of simplicity, we omit `almost surely' when talking about conditional expectations in the sequel. Given a commuting filtration $\indzd\filF i$, the corresponding filtration $\filF\topp q = \{\filF_\ell\topp q\}_{\ell\in\Z}$ defined by
\[
\filF_\ell\topp q = \bigvee_{{i\in\Zd,  i_q \leq \ell}}\filF_i,\quad \ell\in\Z, \quad q=1,\dots,d,
\]
are commuting in the following sense: 
for all permutation $\pi$ of $\{1,\dots,d\}$ and bounded random variable $Y$,
\equh\label{eq:commuting}
\esp\ccbb{\cdots\esp\bb{\esp\pp{Y\mmid \filF_{i_{\pi(1)}}\topp {\pi(1)}}\mmid \filF_{i_{\pi(2)}}\topp {\pi(2)}}\cdots\mmid \filF_{i_{\pi(d)}}\topp {\pi(d)}} = \esp(Y\mid\filF_i),  
\eque
for all $i\in\Zd$ \citep[p.~36, Corollary 3.4.1]{khoshnevisan02multiparameter}. A commuting filtration $\{\filF_i\}_{i\in\Z_+^d}$ with $\Z_+ = \{0\}\cup\N$ is defined similarly.

Given a commuting filtration $\indzd\filF i$, we have $\filF_i = \bigvee_{j\leq i}\filF_j$, and this can be naturally extended to $i\in(\Z\cup\{\infty\})^d$; for example, we write $\filF_{\ell,\infty,\cdots,\infty} = \filF\topp1_\ell$.
We write
\[
\esp_j(\cdot) = \esp(\cdot\mid\filF_j), j\in\Zd \mand \esp\topp q_\ell(\cdot) = \esp(\cdot\mid\filF_\ell\topp q), q=1,\dots,d,\ell\in\Z.
\]
\begin{Def}
A collection of random variables $\{M_n\}_{n\in\Nd}$ is said to be an orthomartingale with respect to a commuting filtration $\{\filF_i\}_{i\in\Z_+^d}$, if for all $n\in\Nd$, $M_n$ is $\filF_n$-measurable, $\esp|M_n|<\infty$, and 
\[
\esp(M_j\mid\filF_i) = M_i \mfa i,j\in\Z_+^d, i\leq j.
\]
We set $M_n\equiv 0$ if $\min\{n_1,\dots,n_d\} = 0$.
\end{Def}
Equivalently, given a commuting filtration $\{\filF_i\}_{i\in\Z_+^d}$, a collection of random variables $\{M_n\}_{n\in\Nd}$ forms an orthomartingale if for each $q=1,\dots,d$, and for all $\{n_\ell\}_{\ell\neq q}\subset\N$ fixed, $n_q\mapsto M_n$ is a one-parameter martingale with respect to the filtration $\filF\topp q$ \citep[p.~37, Theorem 3.5.1]{khoshnevisan02multiparameter}. That is, 
\[
\esp\topp q_{n_q-1}M_n = M_{n-e_q} \mfa n\in\Nd, q=1,\dots,d.
\] 

Given an orthomartingale $\{M_n\}_{n\in\Nd}$ with respect to a commuting filtration $\{\filF_i\}_{i\in\Z_+^d}$, it can be represented as 
\[
M_n = \sum_{i\in[n]}D_i, n\in\Nd
\]
for some $\{D_n\}_{n\in\Nd}$, which are referred to as the orthomartingale differences. When $\{D_n\}_{n\in\Nd}$ is strictly stationary, we say the orthomartingale is stationary. 
Clearly, for all $n\in\Nd$, $D_n$ is $\filF_n$-measurable and 
\[
 \esp\topp q_{n_q-1}D_n = 0, q=1,\dots,d.
\]

Finally, we introduce the {\it projection operators} with respect to a commuting filtrations $\indzd \filF i$ defined by
\equh\label{eq:P0}
P_j = \prodd q1d P\topp q_{j_q}, j\in\Zd
\eque
with $P_\ell\topp q:L^1(\filF)\to L^1(\filF)$ given by
\equh\label{eq:projection}
P\topp q_\ell f = \esp\topp q_\ell f - \esp \topp q_{\ell-1} f, f\in L^1(\filF), \ell\in\Z, q=1,\dots,d.
\eque
Here and in the sequel, for any $\calG\subset\filF$, $L^p(\calG)$  denotes the $L^p$ space $L^p(\Omega,\calG,\proba)$. 

\begin{Example}
When $d=1$, $P_j = \esp_j - \esp_{j-1}$
has been well studied. When $d=2$, 
\[
P_{j_1,j_2}f = \esp_{j_1,j_2}f - \esp_{j_1,j_2-1}f - \esp_{j_1-1,j_2}f + \esp_{j_1-1,j_2-1}f.
\]
\end{Example}

\begin{Lem}\label{lem:projection}
Let $\indzd\filF i$ be a commuting filtration satisfying~\eqref{eq:commuting} and $P_\ell\topp q$ and $P_j$ be defined as above. Then,

(i) $\{P_\ell\topp q\}_{\ell\in\Z, q=1,\dots,d}$ are commuting operators, and so are $\{P_j\}_{j\in\Zd}$. 

(ii) For all $f,g\in L^2(\filF)$, $\esp P_i(f)P_j(g) = 0$ for all $i,j\in\Zd, i\neq j$. 

(iii) For all $f\in L^2(\filF)$, 
\equh\label{eq:Plq}
P_\ell\topp q(f) \in L^2(\filF_\ell\topp q) \ominus L^2(\filF_{\ell-1}\topp q), q=1,\dots,d, \ell\in\Z,
\eque
and
\[
P_j(f) \in \bcapp q1d \pp{L^2(\filF_{j_q}\topp q)\ominus L^2(\filF_{j_q-1}\topp q)}, j\in\Zd.
\]

(iv) For all $i,j\in\Zd,i\neq j$, and $f\in L^2(\filF)$, $P_iP_jf = 0$, almost surely. 
\end{Lem}
\begin{proof}
(i) It suffices to show that for all $f\in L^2(\filF), \ell_1,\ell_2\in\Z, q_1,q_2\in\{1,\dots,d\}, q_1\neq q_2$,
\[
\esp\bb{\esp\pp{ f\mmid \filF_{\ell_1}\topp{q_1}} \mmid \filF_{\ell_2}\topp{q_2}} = \esp\bb{\esp\pp{ f\mmid \filF_{\ell_2}\topp{q_2}} \mmid \filF_{\ell_1}\topp{q_1}}.
\]
This follows from~\eqref{eq:commuting} and~\eqref{eq:projection}.

(ii) Since $i\neq j$, without loss of generality, assume $i_1>j_1$. Then,
\[
\esp (P_i(f)P_j(g)) = \esp\esp\pp{ P_i(f)P_j(g)\mmid \filF_{j_1}\topp1} = \esp \bb{P_j(g)\esp\pp{P_i(f)\mmid\filF_{j_1}\topp1}},
\]
where the last step follows from the fact that $P_j(g)$ is  $\filF_{j_1}\topp1$-measurable. By (i), $P_i(f)$ can be written as $P_{i_1}(g)$ for some $g\in L^2(\filF)$. Thus, since $\esp\spp{P_{i_1}(g)\mid \filF_{i_1-1}\topp1} = 0$ and $i_1-1\geq j_1$, we have $\esp(P_i(f)\mid\filF_{j_1}\topp1) = 0$ and thus the desired orthogonality.

(iii) The fact~\eqref{eq:Plq} follows from the definition. The other statement follows again from the commuting property. 

(iv) It follows from (iii).
\end{proof}

One can generate a filtration $\indzd\filF i$ from a collection of commuting transformations. Namely, let $\{T_{e_q}\}_{q=1,\dots,d}$ be $d$ bijective, bi-measurable and measure-preserving transformations on $(\Omega,\filF,\proba)$, and assume in particular they are commuting: $T_{e_q}\circ T_{e_{q'}} = T_{e_{q'}}\circ T_{e_q}$ for all $q,q'=1,\dots,d$. Then, $\{T_{e_q}\}_{q=1,\dots,d}$ generate a $\Zd$-group of transformations $\indzd Ti$ on $(\Omega,\filF,\proba)$. Let $\calM\subset\filF$ be a $\sigma$-algebra on $\Omega$ such that $\calM\subset T_{e_q}\inv\calM$ for $q=1,\dots,d$. In this way, $\filF_i \defe T_i\inv\calM, i\in\Zd$ yield a filtration. 

However, the filtrations obtained this way are not always commuting.  
\begin{Def}\label{def:complete}
Under the above notations, if in addition~\eqref{eq:commuting1} holds (i.e.~$\indzd\filF i$ are commuting), then $\{T_{e_q}\}_{q=1,\dots,d}$ are said to be {\it completely commuting} with respect to $\calM$ and $\proba$.
\end{Def}

A tightly related concept of complete commutativity has already been discussed by \citet{gordin09martingale}. In our setting this is a property that depends not only on transformations, but also on the specified $\sigma$-algebra and underlying probability space. In the rest of the paper, whenever the transformations $\indzd Ti$ are involved, they are always assumed to be completely commuting and we do not mention specifically $\calM$ and $\proba$, when it is clear from the context, for the sake of simplicity.

Given completely commuting transformations $\indzd Ti$, we consider stationary random fields of the form 
\[
\indzd Xi \equiv \{f\circ T_i\}_{i\in\Zd}
\]
for some function $f$ in the space $L^2(\filF)$. 
In particular, for any $f\in L^2(\filF)$, 
\[
\ccbb{(P_{\vv0}f)\circ T_i}_{i\in\Nd}
\]
gives a collection of stationary orthomartingale differences with respect to $\{\filF_i\}_{i\in\Z_+^d}$.
We also write $U_if = f\circ T_i, i\in\Zd$. Then,
one can readily verify that $U_jP_i = P_{i+j}U_j$ for all $i,j\in\Zd$. This identity will be useful in the sequel.

We conclude this section with two canonical examples for stationary orthomartingales. 
Our main result in Section~\ref{sec:wip} is based on the first example.
\begin{Example}\label{example:bernoulli}
Let $\indzd\epsilon i$ be a collection of independent and identically distributed random variables with distribution $\mu$. We will consider stationary fields as functions of $\indzd\epsilon i$. For this purpose, assume that the probability space $(\Omega,\filF,\proba)$ has the following form
\equh\label{eq:space}
(\Omega,\filF,\proba) \equiv \pp{\R^\Zd,\calB^\Zd,\mu^\Zd},
\eque
and we identify $\epsilon_i(\omega) = \omega_i, i\in\Zd$. 
Let $\indzd Ti$ be the $\Zd$-group of shift operations of $\R^{\Zd}$, so that $[T_i(\omega)]_j = \epsilon_{j+i}, i,j\in\Zd$. It is straight-forward to check that
random variables $\indzd\epsilon i$ induce a commuting filtration $\indzd\filF i$ by
\equh\label{eq:filF}
\filF_i = \sigma\ccbb{\epsilon_j: j\in\Zd, j\leq i}, i\in\Zd.
\eque
\end{Example}
\begin{Example}\label{example:product}
Let $\{\epsilon\topp q_i\}_{i\in\Z}, q=1,\dots,d$ be $d$ independent collections of i.i.d.~random variables. Consider $\calG\topp q_i = \sigma(\epsilon\topp q_j:j\leq i), i\in\Z, q=1,\dots,d$, and set $\filF_i = \bigvee_{q=1}^d\calG\topp q_{i_q}, i\in\Zd$. Clearly this yields a commuting filtration and there is a natural class of completely commuting transformations. 
\end{Example}
\begin{Rem}
These two examples constructing multiparameter filtrations date back at least to the early 70s. See for example~\citet[Section 1]{cairoli75stochastic}, where the commuting filtrations and multiparameter martingales are discussed in the continuous-time setup. 
\end{Rem}
\comment{
\begin{Rem}
Results in Sections~\ref{sec:UI} and~\ref{sec:coboundary} apply to both two examples above. However, our main result on the invariance principle requires the specific construction of the underlying probability as in Example~\ref{example:bernoulli}. An invariance principle of a different type holds for Example~\ref{example:product}, as pointed out by \citet{wang13new}.
\end{Rem}}


\section{A uniform integrability result}\label{sec:UI}
In this section, we establish a uniform integrability result for stationary orthomartingales (Lemma~\ref{lem:UI}). This entails that the tightness of normalized stationary orthomartingales only requires finite second moment (Proposition~\ref{prop:tightness}), thus improving the result of \citet{wang13new}.  The results hold for general orthomartingales with respect to commuting filtrations. The notion of complete commuting is not needed in this section.

In the sequel, we will apply Cairoli's maximal inequality~\citep[p.~19, Theorem 2.3.1]{khoshnevisan02multiparameter} repeatedly:
for an orthomartingale $\{M_n\}_{n\in\Nd}$ with respect to a commuting filtration $\{\filF_n\}_{n\in\Nd}$, 
\equh\label{eq:cairoli}
\esp\pp{\max_{i\in[n]}|M_i|}^p \leq \pp{\frac p{p-1}}^{dp}\esp |M_n|^p,
\eque
for $p>1$. 
To simplify the notation we write $\esp|Z|^p \equiv \esp(|Z|^p), \esp Z^{2k} \equiv \esp(Z^{2k})$, and for $a>0$, $\esp_aY^2 = \esp (Y^2\indd{Y^2>a})$. 
\begin{Lem}\label{lem:UI}
Let $\{M_n\}_{n\in\Nd}$ be a stationary orthomartingale with respect to a commuting filtration $\{\filF_n\}_{n\in\Z_+^d}$. Suppose $\esp D_{\vv1}^2<\infty$. Then, 
\equh\label{eq:UI}
\lim_{a\to\infty} \limsupn \esp_a{\pp{\max_{i\in[n]}\frac{|M_i|}{|n|^{1/2}}}^2} = 0.
\eque
\end{Lem}
\begin{proof}
Recall that $\{D_n\}_{n\in\Nd}$ are stationary orthomartingale differences. For each $i\in\Nd$, define
\[
D_i(c) = P_i(D_i\indd{|D_i|\leq c}) \qmand R_i(c) = D_i - D_i(c).
\]
Clearly, $\{D_n(c)\}_{n\in\Nd}$ and $\{R_n(c)\}_{n\in\Nd}$ are still stationary orthomartingale differences and we write the corresponding orthomartingales by $\{M_n(c)\}_{n\in\Nd}$ and $\{M_n'(c)\}_{n\in\Nd}$, respectively. Then, 
\begin{multline}
\esp_a{\pp{\max_{i\in[n]}\frac{|M_i|}{|n|^{1/2}}}^2} \\
\leq 4\esp_{a/4}{\pp{\max_{i\in[n]}\frac{|M_i(c)|}{|n|^{1/2}}}^2} + 4\esp_{a/4}{\pp{\max_{i\in[n]}\frac{|M_i'(c)|}{|n|^{1/2}}}^2}.\label{eq:truncation}
\end{multline}
Now, the first term on the right-hand side above can be bounded by
\begin{multline*}
4\bb{\esp\pp{\max_{i\in[n]}\frac{|M_i(c)|}{|n|^{1/2}}}^4}^{1/2}\times\proba^{1/2}\bb{\pp{\max_{i\in[n]}\frac{|M_i(c)|}{|n|^{1/2}}}^2>a/4}\\
\leq\frac4{|n|}\bb{\esp \pp{\max_{i\in[n]}|M_i(c)|}^4}^{1/2}\times\pp{\frac4a}^{1/2}\bb{\esp\pp{\max_{i\in[n]}\frac{|M_i(c)|}{|n|^{1/2}}}^2}^{1/2},
\end{multline*}
which, by applying Cairoli's inequality~\eqref{eq:cairoli} twice, can be bounded by
\equh\label{eq:4th}
\frac4{|n|}\pp{\frac43}^{2d}(\esp M_n^4(c))^{1/2}\times \pp{\frac{2^{2d+2}}a}^{1/2}\frac{(\esp M_n^2(c))^{1/2}}{|n|^{1/2}},
\eque
where the second term is bounded by $c(2^{2d+2}/a)^{1/2}$. The first term of~\eqref{eq:4th} can be bounded  by $Kc^2$ for some constant $K$ depending only on $d$ via Burkholder's inequality. To see this,
first we observe that
\[
\ccbb{\sum_{i_1=1}^{n_1}\cdots\sum_{i_{d-1}=1}^{n_{d-1}}D_{i_1,\dots,i_{d-1},i_d}}_{i_d\in\N}
\]
is a sequence of stationary martingale differences with respect to $\indn{\filF\topp d}$. 
Thus, Burkholder's inequality tells, for $p\geq2$, 
\[
\nn{M_n}_p = \nn{\summ {i_1}1{n_1}\cdots\summ{i_d}1{n_d}D_i}_p \leq C_p n_d^{1/2}\nn{\summ {i_1}1{n_1}\cdots\summ{i_{d-1}}1{n_{d-1}}D_i}_p.
\]
Repeating this argument, one obtains that
\[
\esp |M_n|^p \leq C_p^{dp} |n|^{p/2}\esp |D_{\vv 1}|^p.
\]
So the first term on the right-hand side of~\eqref{eq:truncation} can be bounded by $Kc^3/a^{1/2}$ for some constant $K$ depending only on $d$.

Next, the second term on the right-hand side of~\eqref{eq:truncation} can be bounded by
\begin{multline*}
4\esp\pp{\max_{i\in[n]}\frac{|M_i'(c)|}{|n|^{1/2}}}^2
\\
\leq 2^{2+2d}\esp\pp{P_{\vv1}D_{\vv1}\indd{|D_{\vv1}|>c}}^2
\leq 2^{2+2d}\esp(D_{\vv1}^2\indd{|D_{\vv1}|>c}).
\end{multline*}
Combing all above, the desired result~\eqref{eq:UI} follows.
\end{proof}
An immediate consequence of Lemma~\ref{lem:UI} is the tightness of normalized stationary orthomartingales. For $t\in\Rd,n\in\Nd$, we write $t\cdot n = (t_1n_1,\dots,t_dn_d)$ and $M_t = M_{\floor t}$. 
\begin{Prop}\label{prop:tightness}
Under the assumption of Lemma~\ref{lem:UI}, 
\[
\ccbb{\frac{M_{t\cdot n}}{|n|^{1/2}}}_{t\in[0,1]^d}
\]
is tight in $D[0,1]^d$. That is, for all $\epsilon>0$,
\[
\lim_{\delta\downarrow0}\limsupn\proba\pp{\sup_{\substack{s,t\in[0,1]^d\\ \nn{s-t}_\infty<\delta}}|M_{s\cdot n} - M_{t\cdot n}|>\epsilon} = 0.
\]
\end{Prop}
\begin{proof}
For each $\delta\in(0,1), n\in\Nd$, write $\delta n = (\delta n_1,\dots,\delta n_d)$. 
Observe that
\begin{multline*}
p_n(\epsilon,\delta) \defe \proba\pp{\sup_{\substack{s,t\in[0,1]^d\\ \nn{s-t}_\infty<\delta}}|M_{s\cdot n} - M_{t\cdot n}|>3^d\epsilon} \\
\leq \sum_{i\in\{0,\dots,\floor{1/\delta}\}^d}\proba\pp{\sup_{t\in[\delta]^d}\frac{|M_{(\delta i)\cdot n} - M_{(\delta i + t)\cdot n}|}{|n|^{1/2}}>\epsilon}
\\
\leq \frac{2^d}{\delta^d}\proba\pp{\max_{i\in[\delta n]}\frac{|M_i|}{|n|^{1/2}}>\epsilon}.
\end{multline*}
Now,
\begin{multline*}
\proba\pp{\max_{i\in[\delta n]}\frac{|M_i|}{|n|^{1/2}}>\epsilon} \leq \frac1{\epsilon^2}\esp_{\epsilon^2} \pp{\max_{i\in[\delta n]}\frac{|M_i|}{|n|^{1/2}}}^2 \\
= \frac{\delta^d}{\epsilon^2}\esp_{\epsilon^2/\delta^{d}}\pp{\max_{i\in[\delta n]}\frac{|M_i|}{|\delta n|^{1/2}}}^2.
\end{multline*}
So, 
\[
\limsupn p_n(\epsilon,\delta)\leq \frac2{\epsilon^2}\limsupn \esp_{\epsilon^2/\delta^{d}}\pp{\max_{i\in[n]}\frac{|M_i|}{|n|^{1/2}}}^2.
\]
The proof is completed by applying~\eqref{eq:UI}.
\end{proof}
\section{Orthomartingale coboundary representation}\label{sec:coboundary}
In this section, we extend the notion of martingale coboundary representation \citep{gordin69central,heyde75central,volny93approximating} to orthomartingales.
For $S\subset\{1,\dots,d\}$, write $S^c = \{1,\dots,d\}\setminus S$. We assume the commuting filtrations are generated by certain completely commuting transformations $\indzd Ti$ as described in Definition~\ref{def:complete}.
 \begin{Prop}\label{prop:coboundary}
For $f\in L^2(\filF)$ satisfying, for some $M\in\N$, 
\equh\label{eq:M}
\esp\topp q_{-M}f = 0 \qmand f-\esp_M\topp qf = 0  \qmfa q =1,\dots,d,
\eque one can write
\equh\label{eq:coboundary}
f = \sum_{S\subset\{1,\dots,d\}}\prod_{q\in S^c}(I - U_{e_q})h_S
\eque
for some functions $\{h_S\}_{S\subset\{1,\dots,d\}}$, with the convention $\prod_{q\in\emptyset}(I - U_{e_q}) \equiv I$, satisfying for each $S\subset\{1,\dots,d\}$, 
\equh\label{eq:hSq}
h_S\in\bigcap_{q\in S} L^2(\filF\topp q_0)\ominus L^2(\filF\topp q_{-1}),
\eque
and
\equh\label{eq:hd}
h_{\{1,\dots,d\}} = \sumzd j P_{\vv 0}U_jf.
\eque
\end{Prop}
The property~\eqref{eq:hSq} tells that $\{U_{e_q}^kh_S\}_{k\in\N}$ forms a sequence of stationary martingale differences with respect to $\{\filF_n\topp q\}_{n\in\N}$ for $q\in S$.
The explicit formula of $h_S$ is given below in~\eqref{eq:hS}.
\begin{Example}
In the case $d=1$,~\eqref{eq:coboundary} reads as
\[
f = h_{\{1\}}+(I-U)h_\emptyset,
\]
which is the coboundary decomposition in dimension one. 
In the case $d=2$,~\eqref{eq:coboundary} reads as
\equh\label{eq:coboundary2}
f = h_{\{1,2\}} + (I - U_{1,0})h_{\{2\}} + (I - U_{0,1})h_{\{1\}} + (I - U_{0,1})(I - U_{1,0})h_\emptyset,
\eque
where $m$  is an orthomartingale difference with respect to $\{\filF_{i,j}\}_{(i,j)\in\Z^2}$, and $h_{0,1}$ and  $h_{1,0}$  are martingale differences with respect to $\{\filF_{\infty,j}\}_{j\in\Z}$, $\{\filF_{i,\infty}\}_{i\in\Z}$, respectively. 
\end{Example}
\begin{Rem}
Assumption~\eqref{eq:M} is enough for  our purpose in the next section. Here we do not pursue a necessary and sufficient condition for the orthomartingale coboundary decomposition, as did in one dimension by \citet{volny93approximating}. This would require more involved calculations and will be addressed elsewhere. A closely related recent result has been obtained by \citet{gordin09martingale}, who investigated the coboundary representation for {\it reversed} orthomartingales.
\end{Rem}

\begin{proof}[Proof of Proposition~\ref{prop:coboundary}]
We construct $h_S, S\subset\{1,\dots,S\}$ by induction.
For $i\in\Z$, write $v(i) = \indd{i< 0}$. Write $\Z_1 = \{i\in\Z:i\geq 0\}$ and $\Z_0 = \{i\in\Z:i<0\}$. Define two operators $A_{e_q}$ and $B_{e_q}$ by
\[
A_{e_q}f = \sumZ i P_0\topp qU_{e_q}^i f
\]
and
\[
B_{e_q}f = \sumZ i (-1)^{v(i)+1}\sum_{k\in \Z_{v(i)}}P_i\topp qU_{e_q}^kf.
\]
Clearly,  for $f$ satisfying the assumption~\eqref{eq:M}, $A_{e_q}f$ and $B_{e_q}f$ are both well defined, and both as elements in $L^2(\filF)$ satisfy~\eqref{eq:M}. It thus follows that compositions of operators $A_{e_q}$ and $B_{e_q}$ (e.g.~\eqref{eq:hS} below) are well defined for functions satisfying~\eqref{eq:M}.
Observe also that all the pairs of operators $(A_{e_q},A_{e_{q'}}), (A_{e_q},B_{e_{q'}})$ and $(B_{e_q},B_{e_{q'}})$ are commuting for $q\neq q'$ by definition. 

Now, we show that in  an orthomartingale coboundary representation of a function $f$ under~\eqref{eq:M}, one can choose $h_S$ in~\eqref{eq:coboundary} as
\equh\label{eq:hS}
h_S = \prod_{r\in S}A_{e_r}\prod_{s\in S^c}B_{e_s}f.
\eque

The formula~\eqref{eq:coboundary} with~\eqref{eq:hS} is proved by induction. 
In the case $d=1$,~\eqref{eq:coboundary} becomes
\equh\label{eq:d=1}
f = A_{e_1}f + B_{e_1}f - U_{e_1}B_{e_1}f.
\eque
This is the decomposition developed in \citet{volny93approximating}, 
where $A_{e_1}f$ is a martingale difference and $B_{e_1}f - U_{e_1}B_{e_1}f$ is the coboundary (for $m$ and $g$ defined in \citep[p.~45]{volny93approximating}, $m = A_{e_1}f$ and $g=B_{e_1}f$). 
Suppose one has shown for $d-1$ and we now prove the case $d$. For $S\subset\{1,\dots,d-1\}$, write $S^c(d-1) = \{1,\dots,d-1\}\setminus S$. To apply the induction we view $\{\filF_{i_1,\dots,i_{d-1},\infty}\}_{i\in\Z^{d-1}}$ as a $(d-1)$-dimensional commuting filtration, generated by completely commuting transformations $T_{e_1},\dots,T_{e_{d-1}}$ and $\calM = \filF_{0,\dots,0,\infty}$. Thus one has
\[
f = \sum_{S\subset\{1,\dots,d-1\}}g_S \mwith g_S = \prod_{q\in S^c(d-1)}(I - U_{e_q})\prod_{r\in S}A_{e_r}\prod_{s\in S^c(d-1)}B_{e_s}f.
\]
For each $S\subset\{1,\dots,d-1\}$, we apply the one-dimensional martingale coboundary decomposition to~$g_S$, with respect to the filtration $\{\filF\topp d_i\}_{i\in\Z}$. 
Indeed, now~\eqref{eq:d=1} becomes $g_S = A_{e_d}g_S + (I - U_{e_d})B_{e_d}g_S$ with
\eqnh
A_{e_d}g_S 
& = & A_{e_d}\prod_{q\in S^c(d-1)}(I - U_{e_q})\prod_{r\in S}A_{e_r}\prod_{s\in S^c(d-1)}B_{e_s}f
\\
& = & \prod_{q\in S^c(d-1)}(I - U_{e_q})A_{e_d}\prod_{r\in S}A_{e_r}\prod_{s\in S^c(d-1)}B_{e_s}f
\\
&  = & \prod_{q\in (S\cup\{d\})^c}(I-U_{e_q})\prod_{r\in S\cup\{d\}}A_{e_r}\prod_{s\in (S\cup\{d\})^c}B_{e_s}f \\
& = & \prod_{q\in (S\cup\{d\})^c}(I - U_{e_q})h_{S\cup\{d\}},
\eqne
where we used the fact that for any $q\in \{1,\dots,d-1\}$, $A_{e_d}$ and $U_{e_q}$ are commuting, and
\begin{multline*}
(I - U_{e_d})B_{e_d}g_S = \prod_{q\in S^c}(I - U_{e_q})\prod_{r\in S}A_{e_r}\prod_{s\in S^c}B_{e_s}f = \prod_{q\in S^c}(I - U_{e_q})h_S,
\end{multline*}
where we used the fact that for any $q\in \{1,\dots,d-1\}$, $B_{e_d}$ and $A_{e_s}$ are commuting.
Thus,
\begin{multline*}
f = \sum_{S\subset\{1,\dots,d-1\}}\bb{\prod_{q\in (S\cup\{d\})^c}(I - U_{e_q})h_{S\cup\{d\}} + \prod_{q\in S^c}(I - U_{e_q})h_S} \\
= \sum_{S\subset\{1,\dots,d\}}\prod_{q\in S^c}(I - U_{e_q})h_S.
\end{multline*}
It remains to prove~\eqref{eq:hSq} and~\eqref{eq:hd}. Both follow from the construction~\eqref{eq:hS}, and the commuting property of involved operators.
\end{proof}


\section{An invariance principle}\label{sec:wip}

In this section, we prove the main result of the paper. Consider the probability space $(\R^\Zd,\calB_\R^\Zd,\mu^\Zd)$ and the corresponding completely commuting transformations $\indzd Ti$ and filtrations $\indzd\filF i$ as described in Example~\ref{example:bernoulli}. Consider a stationary random field $\{X_i\}_{i\in\Zd}$ in form of
\[
X_i = f\circ T_i(\{\epsilon_k\}_{k\in\Zd}), i\in\Zd.
\]
We consider the following generalized Hannan's condition \citep{hannan73central} for random fields
\equh\label{eq:hannan}
\sum_{i\in\Zd}\nn{P_{\vv0}X_i}_2<\infty.
\eque
Consider partial sums $S_n = \sum_{i\in[n]}, n\in\N$.
Under~\eqref{eq:hannan}, there exists $D_0\in L^2$ such that
\[
\sum_{i\in\Zd}P_{\vv0}X_i \mbox{ converges to }D_{\vv0} \mbox{ in }L^2.
\]
That is, for all $m,n\in\Nd$, $\snn{\sum_{i\in[-m,n]}P_{\vv 0}X_i - D_{\vv 0}}_2\to 0$ as $m,n\to\infty$.

\begin{Thm}\label{thm:wip}
Consider a stationary random field $\indzd Xi$ described as above with zero mean. If Hannan's condition~\eqref{eq:hannan} holds, then, 
\[
\ccbb{\frac{S_{\floor{n\cdot t}}(f)}{|n|^{1/2}}}_{t\in[0,1]^d}\weakto \{\sigma\mathbb B_t\}_{t\in[0,1]^d}
\]
as $n\to\infty$ in $D([0,1]^d)$, where $\{\mathbb B_t\}_{t\in[0,1]^d}$ is a standard Brownian sheet and $\sigma^2 = \esp D_{\vv0}^2$.
\end{Thm}
\begin{proof}[Proof of Theorem~\ref{thm:wip}]
The idea of proof is by orthomartingale approximation. 
Now introduce, for each $n\in\Nd$,
\equh\label{eq:Mn}
M_n = \sum_{i\in[n]}D_i \qmwith D_i = U_iD_{\vv0}, i\in\Nd.
\eque
One easily sees that $\{M_n\}_{n\in\Nd}$ is a $d$-parameter orthomartingale with respect to the filtration $\indzd\filF i$. 

It has been established that for an orthomartingale $\{M_n\}_{n\in\Nd}$ with stationary orthomartingale differences with respect to the filtration generated by i.i.d.~random variables,
\equh\label{eq:WW13}
\ccbb{\frac{M_{\floor{n\cdot t}}(f)}{|n|^{1/2}}}_{t\in[0,1]^d}\weakto \{\sigma\mathbb B_t\}_{t\in[0,1]^d}
\eque
in $D([0,1]^d)$. For the convergence of finite-dimensional distributions, see \citet{wang13new}; for the tightness under finite second moment, see Proposition~\ref{prop:tightness}. 
Thus, it suffices to show, 
\equh\label{eq:Sn-Mn}
\limn\proba\pp{\frac1{|n|^{1/2}}\max_{m\in[n]}|S_m(f) - M_m|>\epsilon} = 0 \mfa \epsilon>0.
\eque
To do so, observe that $f\in L^2(\filF)$ and the fact that $\indzd\epsilon i$ are i.i.d.~imply
\equh\label{eq:f}
f = \sum_{i\in\Zd}P_if
\eque
where the summation converges in $L^2$, and introduce 
\[
f\topp k = \sum_{i\in[-k,k]}P_if 
\]
and
\[
M_n\topp k = \sum_{i\in[n]}U_iD_{\vv0}\topp k \qmwith D_{\vv0}\topp k = \sum_{i\in[-k,k]}P_{\vv0}U_if.
\]
Then,
\begin{multline}\label{eq:3terms}
\max_{m\in[n]}|S_m(f) - M_m| \\
\leq \max_{m\in[n]}|S_m(f) - S_m(f\topp k)|  + \max_{m\in[n]}|S_m(f\topp k) - M_m\topp k| + \max_{m\in[n]}|M_m\topp k-M_m|.
\end{multline}
We control the three maxima separately. 

(i) To estimate the first term on the right-hand side of~\eqref{eq:3terms}, we need the following maximal inequality.  
\begin{Lem}\label{lem:wu}
Under the assumption of Theorem~\ref{thm:wip}, for all $n\in\Nd$,
\equh\label{eq:maximal}
\nn{\max_{m\in[n]}S_m(f)}_2 \leq 2^d|n|^{1/2}\sum_{i\in\Zd}\nn{f_i}_2,
\eque
with  $f_i = P_{\vv0}U_if \in L^2_{\vv0}, i\in\Zd$.
\end{Lem}
The proof is postponed to the end of section. Now,~\eqref{eq:maximal} yields
\begin{multline}\label{eq:maximal1}
\proba\pp{\frac1{|n|^{1/2}}\max_{m\in[n]}|S_m(f) - S_m(f\topp k)|>\epsilon} \\
\leq \pp{\frac{2^d}{\epsilon}\sumzd i\nn{(f-f\topp k)_i}_2}^2.
\end{multline}
Since $U_iP_j = P_{j+i}U_i, i,j\in\Zd$, observe that
\[
(f-f\topp k)_i = P_{\vv 0}U_i\pp{\sum_{j\notin[-k,k]}P_jf} = P_{\vv0}\sum_{j\notin[-k,k]}P_{j+i}U_if = f_i \indd{i\notin[-k,k]}.\]
Thus, by taking $\min\{k_1,\dots,k_d\}$ large enough, the upper bound in~\eqref{eq:maximal1} can be arbitrarily small.

(ii) To estimate the last term in the right-hand side of~\eqref{eq:3terms}, observe that $\{M_n-M_n\topp k\}_{n\in\Nd}$ is still a stationary orthomartingale. Again by Cairoli's maximal inequality, we have
\equh\label{eq:maximal3}
\proba\pp{\frac1{|n|^{1/2}}\max_{m\in[n]}\abs{M_m - M_m\topp k}>\epsilon} \leq \pp{\frac{2^d}{\epsilon}\nn{D_{\vv 0} - D_{\vv 0}\topp k}_2}^2. 
\eque
Thus, by taking $\min\{k_1,\dots,k_d\}$ large enough, the upper bound in~\eqref{eq:maximal3} can be arbitrarily small.

(iii) At last, write
\[
S_m(f\topp k) - M_m\topp k = \sum_{i\in[m]}U_i(f\topp k - D_{\vv0}\topp k).
\]
It remains to show that
\equh\label{eq:maximal2}
\lim_{k\to\infty}\limsupn\proba\pp{\frac1{|n|^{1/2}}\max_{m\in[n]}\abs{\sum_{i\in[m]}U_i(f\topp k - D_{\vv0}\topp k)}>\epsilon} = 0.
\eque
By Proposition~\ref{prop:coboundary}, $f\topp k - D_{\vv0}\topp k$ has an orthomartingale coboundary representation~\eqref{eq:coboundary}, and in particular,~\eqref{eq:hd} becomes
\eqnh
h_{\{1,\dots,d\}} & = & P_{\vv 0}\sum_{j\in\Zd}U_j(f\topp k - D_{\vv0}\topp k) \\
& = & P_{\vv0}\sumzd jU_j\sum_{\ell\in[-k,k]}P_\ell f - P_{\vv0}\sumzd jU_j\sum_{\ell\in[-k,k]}P_{\vv0}U_\ell f\\
& = & P_{\vv0}\sum_{\ell\in[-k,k]}U_\ell f - P_{\vv0}\sum_{\ell\in[-k,k]}U_\ell f  =  0.
\eqne
Thus,
\equh\label{eq:coboundaryapplication}
f\topp k - D_{\vv 0}\topp k = \sum_{S\subsetneq\{1,\dots,d\}}\prod_{q\in S^c}(I - U_{e_q})h_S.
\eque
To prove~\eqref{eq:maximal2}, it suffices to show for each $S\subsetneq\{1,\dots,d\}$, 
\equh\label{eq:maximal2'}
\limn\proba\pp{\max_{m\in[n]}\frac{\sum_{i\in[m]}U_i(\prod_{q\in S^c}(I - U_{e_q})h_S)}{|n|^{1/2}}>\epsilon} = 0.
\eque
To better illustrate, we first prove the case $d=2$. Suppose $S = \{1\}$. Notice that $U_{1,0}^k = U_{k,0}, k\in\Z$ by definition, and similarly for $U_{0,1}$. Then, for $n\in \N^2$,
\begin{multline}\label{eq:S=1}
\max_{m\in[n]}\sum_{i\in[m]}U_i\prod_{q\in S^c}(I - U_{e_q})h_S \\
= \max_{m\in[n]}\sum_{i\in [m]}U^{i_1}_{1,0}(U^{i_2}_{0,1}(I - U_{0,1})h_1)
= \max_{m\in[n]}\summ{i_1}1{m_1}U^{i_1}_{1,0}(U_{0,1} - U_{0,m_2+1})h_1\\
\leq 2\max_{m_2=1,\dots,n_2+1}\abs{U_{0,m_2}\max_{m_1 = 1,\dots,n_1}\summ {i_1}1{m_1}U_{i_1,0}h_1}.
\end{multline}
Write 
\[
\wt M_{m_1} = \sum_{i_1=1}^{m_1}U_{1,0}^{i_1}h_1.
\]
Observe that by Proposition~\ref{prop:coboundary}, $\{U_{1,0}^{i_1}h_1\}_{i_1\in\N}$ is a sequence of stationary martingale differences with respect to the filtration $\{\filF_n\topp 1\}_{n\in\N}$. 
So, the probability in~\eqref{eq:maximal2'} is bounded by
\begin{multline}\label{eq:s<d}
(n_2+1)\proba\pp{\frac{\sabs{\max_{m_1\leq n_1}\wt M_{m_1}}}{(n_1n_2)^{1/2}}>\epsilon/2}
\\
\leq (n_2+1) \esp_{\frac{\epsilon^2}4}\pp{\frac{\max_{m_1\leq n_1}\wt M_{m_1}^2}{n_1n_2}} 
\leq \frac{n_2+1}{n_2\epsilon^2}\esp_{n_2\frac{\epsilon^2}4}\pp{\frac{\max_{m_1\leq n_1}\wt M_{m_1}}{n_1^{1/2}}}^2.
\end{multline}
By uniform integrability~\eqref{eq:UI}, the last term above tends to zero as $\min(n_1,n_2)\to\infty$. 

The same argument applies to the case $S = \{2\}$.
For the case $S = \emptyset$, the probability in~\eqref{eq:maximal2'} is bounded by
\begin{multline}\label{eq:s=d}
\proba\pp{\frac{\max_{m_1\leq n_1,m_2\leq n_2}U_{m_1,m_2}h_S}{(n_1n_2)^{1/2}}>\epsilon/4} \\
\leq (n_1+1)(n_2+1)\proba\pp{\frac{|h_S|}{(n_1n_2)^{1/2}}>\epsilon/4} \leq \frac{(n_1+1)(n_2+1)}{n_1n_2{\epsilon^2}/{16}}\esp_{n_1n_2{\epsilon^2}/{16}}h_S^2, 
\end{multline}
which tends to zero as $n\to\infty$. We have thus proved~\eqref{eq:maximal2'} for $d=2$.

At last we sketch the proof for general $d\geq 3$. Without loss of generality, we suppose $S^c = \{s+1,\dots,d\}$ with $s=0,\dots,d-1$. In the case $s=0$,~\eqref{eq:s=d} can be easily generalized and we omit the details. In the case $s\geq1$, observe that
\begin{multline*}
\sum_{i\in[m]}U_i\prod_{q\in S^c}(I - U_{e_q})h_S \\
= \sum_{i_{1}=1}^{m_{1}}\cdots\sum_{i_s=1}^{m_s}\prodd q{1}sU_{e_q}^{i_q}\sum_{i_{s+1}=1}^{m_{s+1}}\cdots\sum_{i_d=1}^{m_d}\prodd q{s+1}dU_{e_q}^{i_q}\prodd r{s+1}d(I - U_{e_r})h_S\\
= \sum_{i_{1}=1}^{m_{1}}\cdots\sum_{i_s=1}^{m_s}\prodd q{1}sU_{e_q}^{i_q}\prodd r{s+1}d(U_{e_r}-U_{e_r}^{m_r+1})h_S\\
= \prodd r{s+1}d(U_{e_r}-U_{e_r}^{m_r+1})\pp{\sum_{i_{1}=1}^{m_{1}}\cdots\sum_{i_s=1}^{m_s}\prodd q{1}sU_{e_q}^{i_q}h_S}.
\end{multline*}
Thus,~\eqref{eq:s<d} becomes, for $n\in\Nd$,
\equh\label{eq:s<d'}
\prodd q{s+1}d(n_q+1)\proba\pp{\frac{\max_{m_q\leq n_q,q=1,\dots,s}\wt M_{m_{1},\dots,m_s}}{|n|^{1/2}}>\epsilon/2^{d-s}},
\eque
with
\[
\wt M_{m_{1},\dots,m_s} = \sum_{i_{1}=1}^{m_{1}}\cdots\sum_{i_s=1}^{m_s}\prodd q{1}sU_{e_q}^{i_q}h_S.
\]
By Proposition~\ref{prop:coboundary} again, this time $\{\prodd q{1}sU_{e_q}^{i_q}h_S\}_{i_{1},\dots,i_s\in \N^{s}}$ form a collection of $s$-dimension stationary orthomartingale differences, with respect to the commuting filtration $\{\filF_{n_{1},\dots,n_s,\infty,\dots\infty}\}_{n_{1},\dots,n_s\in \N^{s}}$. Therefore~\eqref{eq:s<d'} can be bounded as before by
\[
\prodd q{s+1}d\frac{n_q+1}{n_q}\esp_{(\prodd q{s+1}d n_q)\epsilon^2/2^{2(d-s)}}\pp{\frac{\max_{m_q\leq n_q,q=1,\dots,s}\wt M_{m_{1},\dots,m_s}}{(\prodd q{1}s n_q)^{1/2}}}^2,
\]
which again tends to zero as $n\to\infty$ by the uniform integrability~\eqref{eq:UI}.
\end{proof}
\begin{Rem}
The approximation of $S_n$ by $M_n$~\eqref{eq:Sn-Mn} actually holds for more general commuting filtrations generated by completely commuting transformations. Then~\eqref{eq:f} may not be guaranteed and  extra assumption on the regularity of the random field will be needed: $X_0$ is $\filF_{\infty,\dots,\infty}$-measurable and $\esp(X_0\mid\filF_i) \to 0$ in $L^2$ whenever $\min_{q=1,\dots,d}i_q\to-\infty$.

However, a crucial ingredient of the proof is the invariance principle~\eqref{eq:WW13} for $M_n$ established by \citet{wang13new}. For this result to hold, our assumption on the underlying random field of i.i.d.~random variables indexed by $\Zd$~ is needed. Without this assumption, in general a stationary orthomartingale difference random field may converge to a limit distribution that is not Gaussian \citep[Example 1]{wang13new}.

 \comment{Consider for example the model in Example~\ref{example:product} with $\esp\epsilon_0\topp q = 0$, $\esp(\epsilon_0\topp q)^2 = 1$. In this case, $D_i = \prodd q1d \epsilon_{i_q}\topp q, i\in\Nd$ form a collection of stationary orthomartingale differences, and for the corresponding orthomartingale $\{M_n\}_{n\in\Nd}$, $|n|^{-d/2}M_n$ converges as $n\to\infty$ in distribution to the product of $d$ independent standard Gaussian random variables.}
\end{Rem}
\begin{proof}[Proof of Lemma~\ref{lem:wu}]
Recall that $f_i = P_{\vv0}U_if$ and~\eqref{eq:f}. Since
\[
S_n(f) = \sum_{j\in[n]}U_j\sum_{i\in\Zd}P_if = \sum_{i\in\Zd}\sum_{j\in[n]}U_jP_if = \sum_{i\in\Zd}\sum_{j\in[n]}U_{j-i}f_i,
\]
we have for all $m\in[n]$,
\begin{multline*}
|S_m(f)| \leq \sumzd i\abs{\sum_{j\in[m]}U_{j-i}f_i}\\
 \leq \sumzd i\max_{k\in[n]}\abs{\sum_{j\in[k]}U_{j-i}f_i}\leq \sumzd iU_{-i}\pp{\max_{k\in[n]}\abs{\sum_{j\in[k]}U_jf_i}}.
\end{multline*}
Therefore, 
\[
\nn{\max_{m\in[n]}S_m(f)}_2 \leq \sumzd i\nn{{\max_{k\in[n]}\abs{\sum_{j\in[k]}U_jf_i}}}_2.
\]
Observe that for each $i$ fixed, $\{\sum_{j\in[k]}U_jf_i\}_{k\in\Nd}$ is an orthomartingale with respect to the filtration $\indzd\filF i$. Therefore, by Cairoli's inequality~\eqref{eq:cairoli},
\[
\nn{\max_{k\in[n]}\abs{\sum_{j\in[k]}U_jf_i}}_2\leq 2^d\nn{S_n(f_i)}_2 = 2^d|n|^{1/2}\nn{f_i}_2,
\]
where in the last step we used the fact that $\{U_jf_i\}_{j\in[n]}$ is a collection of stationary orthomartingale differences. 
\end{proof}
\section{Discussions}\label{sec:discussions}
There are some recent developments on sufficient conditions for central limit theorem and invariance principle of stationary random fields, notably by  \citet{elmachkouri13central} and \citet{wang13new}. We compare our condition to theirs.

We first show that the Hannan's condition is {\it strictly} weaker than Wu's condition \citep{wu05nonlinear,elmachkouri13central}
\equh\label{eq:wu}
\sum_{i\in\Zd}\delta_i(f)<\infty
\eque
where $\delta_i(f)$ is the {\it physical dependence measure} for a stationary random field $\{f\circ T_i\}_{i\in\Zd}$, which we will recall in a moment. \citet{elmachkouri13central} showed that this condition implies central limit theorem for stationary random fields. In dimension one, it has been shown in \citet[Theorem 1]{wu05nonlinear} that~\eqref{eq:wu}
implies Hannan's condition~\eqref{eq:hannan}, and the argument can be easily adapted to high dimension and the details are omitted. We provide an example in Proposition~\ref{prop:comparison} below that satisfies Hannan's condition but not~\eqref{eq:wu}. It suffices to construct a martingale difference random field that violates~\eqref{eq:wu}.

However, we remark also that
the results of \citet{elmachkouri13central} are more general in the sense that they include central limit theorem and invariance principle for random fields indexed by non-rectangular sets. In this case they assume stronger assumption on the moment in terms of entropy of the index sets.

In the sequel, suppose $\epsilon = \{\epsilon_i\}_{i\in\Zd}$ is a sequence of i.i.d.~random variables with $\proba(\epsilon_{\vv0} = \pm1) = 1/2$. Then, for a function $f:\{\pm1\}^\Zd\to\R$, the physical dependence measure is defined by
\equh\label{eq:deltai}
\delta_i = \nn{f(\epsilon) - f(\epsilon^{*i})}_2
\qmwith
\epsilon^{*i}_k = \left\{
\begin{array}{l@{\mbox{ if }}l}
\epsilon_k  & k\neq i\\
\epsilon_i^* & k = i,
\end{array}
\right.
\eque
where $\epsilon^{*i}$ is a copy of $\epsilon_k$, independent of $\epsilon$. 
\begin{Prop}\label{prop:comparison}
Under the above assumption, there exists a martingale difference that does not satisfy~\eqref{eq:wu}.
\end{Prop}
\begin{proof}
We first address the case $d=1$. Set
\eqnh
Z_1(\epsilon) & = & \indd{\epsilon_{-2} = -1}\indd{\epsilon_{-1}=-1}\epsilon_0\\
Z_2(\epsilon) & = & \indd{\epsilon_{-4} = -1}\indd{\epsilon_{-3}=-1}\indd{\epsilon_{-2} = 1}\indd{\epsilon_{-1}=1}\epsilon_0\\
& \cdots & \\
Z_n(\epsilon) & = & \indd{\epsilon_{-2n} = -1}\indd{\epsilon_{-2n+1} = -1}\indd{\epsilon_{-2n+2} = 1}\cdots\indd{\epsilon_{-1}=1}\epsilon_0, n\geq 3.
\eqne
Define
\[
f = f(\epsilon) = \sif n1 c_nZ_n(\epsilon)
\]
for certain sequence of real values $\indn c$ such that 
\equh\label{eq:ck}
\sif n1 c_n^2\nn{Z_n(\epsilon)}_2^2 <\infty.
\eque 
Under this condition, clearly $f$ is well defined and a martingale difference in the sense that $f\in\filF_0$ and $\esp(f\mid\filF_{-1}) = 0$. 

Now we compute $\delta_i$ defined in~\eqref{eq:deltai}. Observe that for $i>0$, $\delta_i = 0$. From now on suppose $i<0$. Suppose $i = -(2k-1)$ or $-2k$ for some $k\in\N$, then we have
\[
f(\epsilon) - f(\epsilon^{*i}) = \sif j{k}c_j(Z_j(\epsilon) - Z_j(\epsilon^{*i})).
\]
Observe that by construction, for all $j\neq j'$, $Z_{j}(\epsilon)Z_{j'}(\epsilon^{*i}) \equiv0$, and
\[
\proba(Z_j(\epsilon)\neq Z_j(\epsilon^{*i})\mid Z_j(\epsilon)\neq 0) = 1/2, \mfa j\geq k.
\]
Thus,
\[
\bb{\sum_{j=k}^\infty c_j(Z_j(\epsilon) - Z_j(\epsilon^{*i}))}^2 = \sif j{k}c_j^2 (Z_j(\epsilon) - Z_j(\epsilon^{*i}))^2,
\]
and for each $j\geq k$,
\eqnh
\esp(Z_j(\epsilon) - Z_j(\epsilon^{*i}))^2 & \geq & \proba(Z_j(\epsilon)\neq 0)\esp\bb{(Z_j(\epsilon)-Z_j(\epsilon^{*i}))^2\mid Z_j(\epsilon)\neq 0}\\
& = & \proba(Z_j(\epsilon)\neq 0)\proba(Z_j(\epsilon)\neq Z_j(\epsilon^{*i})\mid Z_j(\epsilon)\neq 0)\\
& = & \frac12\nn{Z_j(\epsilon)}_2^2.
\eqne
Thus,
\[
\delta_i^2  = \esp \bb{\sif jkc_j(Z_j(\epsilon) - Z_j(\epsilon^{*i}))^2}^2 \geq \frac12\sif jk c_j^2\nn{Z_j(\epsilon)}_2^2,
\]
and
\[
\sum_{i\leq -1}\delta_i^2(f) \geq \sif k1\sif jk c_j^2\nn{Z_j(\epsilon)}_2^2 = \sif j1 jc_j^2\nn{Z_j(\epsilon)}_2^2.
\]
Now, choose $\indn c$ such that $c_n^2\nn{Z_n(\epsilon)}_2^2 = n^{-2}$, so that $f$ is well-defined since~\eqref{eq:ck} is satisfied. However, $\sum_i\delta_i^2(f) = \infty$ whence $\sum_i\delta_i(f)= \infty$, as desired.

It remains to prove the case $d\geq 2$. This can be done by first assigning an ordering of the space $\{i\in\Zd: i\leq-\vv1\}$ and then embedding the one-dimensional construction. The details are omitted.
\end{proof}

Next, our results also improve \citet{wang13new}. They proved a central limit theorem for stationary random field under the condition
\equh\label{eq:WMstrong}
\sum_{k\in\Nd} \frac{\nn{\esp(X_k\mid\filF_{\vv0})}_2}{|k|^{1/2}}<\infty,
\eque
and established an invariance principle under a slightly stronger assumption, replacing $\nn\cdot_2$ by $\nn\cdot_p$ for some $p>2$ in~\eqref{eq:WMstrong}. The Hannan's condition~\eqref{eq:hannan} we assumed here is weaker than~\eqref{eq:WMstrong}. This is known in dimension one, see \citet[Corollary 2]{peligrad06central}. We prove the result for high dimension in Lemma~\ref{lem:WMstrong}.
\begin{Lem}\label{lem:WMstrong}
Condition~\eqref{eq:WMstrong} implies Hannan's condition~\eqref{eq:hannan}.
\end{Lem}
\begin{proof}
For $n\in\Nd$, set $a_n = \nn{P_{\vv0}X_n}_2$. Then, it is equivalent to show
\equh\label{eq:infty}
\sum_{n\in\Nd}a_n = \infty \mbox{ implies } \sum_{n\in\Nd}\frac1{|n|^{1/2}}\pp{\sum_{k\geq n}a_k^2}^{1/2}=\infty.
\eque
To see this, first observe that by orthogonality, for each $n\in\Nd$,
\[
\nn{\esp(X_n\mid\filF_0)}_2^2 = \sum_{k\geq\vv0}\nn{P_{-k}X_n}_2^2 = \sum_{k\geq n}\nn{P_{\vv0}X_k}_2^2 = \sum_{k\geq n}a_k^2.
\]
To prove~\eqref{eq:infty}, introduce $B_n = \{k\in\Nd:n\leq k\leq 2n - \vv 1\}$, and observe
\begin{multline*}
\sum_{n\in\Nd}\frac1{|n|^{1/2}}\pp{\sum_{k\geq n}a_k^2}^{1/2}\geq \sum_{n\in\Nd}\frac1{|n|^{1/2}}\pp{\sum_{k\in B_n}a_k^2}^{1/2} \\
\geq \sum_{n\in\Nd}\frac1{|n|}\sum_{k\in B_n}a_k = \sum_{k\in\Nd}a_k\sum_{n\in\Nd}\frac1{|n|}\indd{k\in B_n}\geq \frac1{2^d}\sum_{k\in\Nd}a_k.
\end{multline*}
\end{proof}
The fact that~\eqref{eq:WMstrong} is actually {\it strictly stronger} than Hannan's condition follows from \citet{durieu08comparison} and \citet{durieu09independence}, in the case $d=1$. Indeed, they constructed a counterexample to show that Hannan's condition does not imply the Maxwell--Woodroofe condition \citep{maxwell00central}, and the latter is known to be strictly weaker than~\eqref{eq:WMstrong}.  Thus, if Hannan's condition implies~\eqref{eq:WMstrong}, it then implies the Maxwell--Woodroofe condition, hence a contradiction. The counterexample therein can be generalized to $\Zd$. 
\bigskip

{\bf Acknowledgement} The authors thank Mohamed El Machkouri and Davide Giraudo for many inspiring discussions. The authors are grateful to an anonymous referee for pointing out an erroneous statement about commuting transformations in the previous version, and for raising our attention to the completely commuting property discussed in \citet{gordin09martingale}. The second author thanks Laboratoire Math\'ematiques Rapha\"el Salem at Universit\'e de Rouen for the invitation during June, 2013, during which the main result of this work was obtained. The second author was partially supported by Faculty Research Grant from the University Research Council at University of Cincinnati in 2013. 

\bibliographystyle{apalike}
\bibliography{references}

\end{document}